\documentclass{amsart}

\usepackage{graphicx}
\usepackage{enumerate,amssymb,amstext,amsmath,hyperref,amsfonts,amscd,amsbsy,amsthm}

\newcommand\ga{\alpha}
\newcommand\gb{\beta}

\newcommand\gd{\delta}
\renewcommand\ge{\epsilon}

\renewcommand\th{\theta}

\newcommand\gw{\omega}
\newcommand\gv{\upsilon}
\newcommand\bgw{\boldsymbol{\omega}}
\newcommand\bgv{\boldsymbol{\upsilon}}

\newcommand\gD{\boldsymbol{\gd}}

\newcommand\bbD{\mathbf{D}}

\newcommand\bT{{\bf T}}

\newcommand\ba{{\bf a}}

\newcommand\bq{{\bf q}}
\newcommand\bbR{{\bf r}}
\newcommand\bs{{\bf s}}

\newcommand\bv{{\bf v}}
\newcommand\bw{{\bf w}}
\newcommand\bx{{\bf x}}
\newcommand\by{{\bf y}}

\newcommand\bzero{\boldsymbol{0}}

\newcommand\x{\times}

\renewcommand\bbR{{\mathbb R}}
\renewcommand\bbD{{\mathbb D}}

\newcommand\Lse{\mathfrak{se}}

\newcommand\Lso{\mathfrak{so}}

\newcommand\Lg{\mathfrak{g}}

\newcommand\Lt{\mathfrak{t}}

\newcommand\Ad{\mathrm{Ad}}

\renewcommand\({\left(}
\renewcommand\){\right)}

\newcommand\dsum{\displaystyle\sum}

\newcommand\nbb{\nabla}

\newcommand\diag{\mathrm{diag}\,}

\renewcommand{\v}[1]{\ensuremath{\mathbf{#1}}}

\newtheorem{thm}{Theorem}[section]
\newtheorem{cor}[thm]{Corollary}
\newtheorem{lem}[thm]{Lemma}

\theoremstyle{definition}
\newtheorem{definition}[thm]{Definition}

\begin{document}

\title{Invariants of the $k$-fold adjoint action of the Euclidean group
}

\author{Mohammed~Daher and Peter~Donelan}
\address{School of Mathematics, Statistics \& Operations Research, Victoria~University~of~Wellington, PO Box 600, Wellington 6140, New Zealand}
\email{peter.donelan@vuw.ac.nz}

\maketitle

\begin{abstract}
A non-zero element of the Lie algebra $\mathfrak{se}(3)$ of the special Euclidean spatial isometry group $SE(3)$ is known as a {\em twist} and the corresponding element of the projective Lie algebra is termed a {\em screw}. Either can be used to describe a one-degree-of-freedom joint between rigid components in a mechanical device or robot manipulator.  This leads to  a practical interest in multiple twists or screws, describing the overall instantaneous motion of such a device. In this paper, invariants of multiple twists under the action induced by the adjoint action of the group are determined.  The ring of the polynomial  invariants  for the adjoint action of $SE(3)$ acting on a single twist is well known to be finitely generated by the Klein and Killing forms, while a theorem of Panyushev~\cite{pchev} gives finite generation for the real invariants  of the induced action on two twists. However we are not aware of a corresponding theorem for $k$~twists, where $k\geq3$.  Following Study~\cite{study}, we use the principle of transference to determine fundamental algebraic invariants and their syzygies. We prove that the ring of invariants  for triple twists is rationally finitely generated by $13$ of these invariants. 
\end{abstract}

\section{Introduction}

\label{s:intro}

Joints in mechanisms and robot manipulators correspond mathematically to subalgebras of the Euclidean Lie algebra and the motions they admit are exponentials of these.  The subalgebras are known in the kinematics literature as screw systems~\cite{hunt}. In the simplest case of a joint with one degree of freedom, these are simply a single screw---the span of any non-zero element of the Lie algebra, known as a twist.  A typical mechanism or manipulator will involve a number of joints so that determining invariants of sets or sequences of twists under the adjoint action provides a means for classification, as well as providing geometric insight into the variety of such devices. The  goal of this paper is to present an approach to obtaining a complete picture of such Euclidean adjoint invariants, with particular emphasis on the case of three twists, via dualisation of classical invariants.

The \textit{special Euclidean group $SE(n)$ of order $n$} is the subgroup of the isometry group of affine $n$-space consisting of the direct isometries, that is, isometries preserving orientation. It is well known that, relative to a choice of orthogonal coordinates identifying affine with Euclidean $n$-space, any such isometry can be nicely identified as a composition of a rotation about the origin---an element of the special orthogonal group $A\in SO(n)$---followed by translation by some vector $\ba\in\bbR^n$. Thus the group is isomorphic to a semi-direct product:
\begin{align}\label{semi}
 SE(n)\cong SO(n)\ltimes \bbR^n.
\end{align}
Correspondingly, the Lie algebra $\Lse(n)$ is isomorphic to a semi-direct sum (so the Lie bracket is not defined component-wise) of the Lie algebra $\Lso(n)$ of skew-symmetric $n\x n$ matrices and the abelian algebra $\Lt(n)$ of infinitesimal translations.   We shall be interested only in the case $n=3$ which has very special properties arising from the identification of the natural action of $SO(3)$ that determines the semi-direct product with its adjoint action. 

A {\em twist} is an element of the Lie algebra~$\Lse(3)$. Twists can be represented in a variety of ways, but the most succinct is using  (generalised) Pl\"ucker coordinates (see, for example,~\cite{selig}).  Making use of the isomorphism between $\Lso(3)$ and $\Lt(3)$ given by:
\begin{equation}
\label{e:skewrep}
\begin{bmatrix}0&-\gw_3&\gw_2\\ \gw_3&0&-\gw_1\\ -\gw_2&\gw_1&0\end{bmatrix}\longleftrightarrow\begin{bmatrix}\gw_1\\ \gw_2\\ \gw_3\end{bmatrix},
\end{equation}
we  denote a twist $\bs$ as a 6--vector consisting of a pair of 3--vectors: $\bs=(\bgw,\bgv)$. Strictly speaking, this should be $(\bgw^t,\bgv^t)^t$, but there should be no ambiguity in the simpler notation. 
Kinematically, a joint in a manipulator with one degree of freedom can be represented by a twist $\bs$ and the motion of one link with respect to its adjoined link is given by $\exp(\th\bs)$ where $\th\in\bbR$ is the joint variable.  These Pl\"ucker coordinates rely on the choice of spatial coordinate frame, but we would expect kinematic properties to be invariant under change of coordinates. 
Mathematically, they should be invariants of the {\em adjoint action} $\Ad$ of the Euclidean group on its Lie algebra.  The most familiar of these invariants is the {\em pitch} of a twist which, in terms of Pl\"ucker coordinates for $\bgw\neq\bzero$, is the ratio:
\begin{equation}
\label{e:pitch}
h=\frac{\bgw.\bgv}{\bgw.\bgw}
\end{equation}
of the Klein form:
\begin{equation}
\label{e:klein}
\bgw.\bgv=\gw_1\gv_1+\gw_2\gv_2+\gw_3\gv_3
\end{equation}
and the Killing form: 
\begin{equation}
\label{e:killing}
\bgw.\bgw=\gw_1^2+\gw_2^2+\gw_3^2.
\end{equation}
Each of these forms is an {\em invariant} polynomial $f(\bgw,\bgv)$ of the adjoint action in the sense that for any $g\in SE(3)$, that is:
\begin{equation}
\label{e:invpoly}
f(\Ad(g)(\bgw,\bgv))=f(\bgw,\bgv).
\end{equation}
We are working here with the adjoint action of the semi-direct product $SO(3)\ltimes\bbR^3$---that is we assume a given coordinate frame in affine 3--space. From now on we shall work under this assumption and therefore identify $SE(3)$ with the product. 

Almost every twist gives rise to an exponential motion having an invariant axis about which it rotates and along which it translates.  The pitch represents the displacement along the axis resulting from one full rotation about it.  Exceptional cases are those for which $\bgw=\bzero$, when the pitch is conventionally set to be $\infty$. On the other hand, if $\bgw\neq\bzero$ and the Klein form vanishes, the pitch is zero and one can identify the twist with its axis---in this case the coordinates correspond with the Pl\"ucker line coordinates of the axis. 

For a serial kinematic chain, we have several joints, and hence are interested in the invariants of a set of twists $\bs_1,\dots,\bs_k$. Since the Euclidean group is algebraic---that is it can be represented as the zero set of polynomials---its polynomial invariants are of particular significance. Other geometrically relevant invariant quantities, such as the pitch, may be expressed rationally or algebraically in terms of them. Our primary goal is to determine fundamental  invariants of the action induced by the adjoint action of $SE(3)$ for sets of three twists and to show that the ring of invariants is rationally finite generated.  In the terminology of Weyl~\cite{weyl}, these are called {\em vector invariants}---that is, the invariants of an action on a `vector' of elements in the space of the action. 

While this paper concentrates on establishing the theory of $k$-fold invariants, their geometric interpretation is, of course, important and is explored further in \cite{daher,dahdon}

A number of authors \cite{gibson,per,ros,tak} have explored invariants of the adjoint and co-adjoint action (the latter of importance in theoretical physics). Selig~\cite{selig} explicitly makes use of the principle of transference to derive invariants that essentially correspond to those we obtain. In~\cite{crook}, the algebraic method of SAGBI bases is employed to find some of these invariants. The work of Takiff~\cite{tak} has been extended by Panyushev~\cite{pchev}, but in a technical algebraic--geometric setting, and he too obtains the dualised form of 2-fold invariants. 

Our guide to find the vector invariants  is the {\em principle of transference}, whose origins are in the work of Clifford and which was subsequently formulated by A~P.~Kotelnikov and E.~Study~\cite{study}.  For more recent descriptions of the principle see, for example, Rico  and Duffy~\cite{duffy}, Chevallier~\cite{chev}, Selig~\cite{selig} and Rooney~\cite{rooney}.  An algebraic version of the principle states that on replacing real coordinates by dual coordinates, valid statements about vectors in $\bbR^3$ become valid statements about twists, written as dual vectors $\bgw+\ge\bgv$.  Here $\ge$ is a quantity such that $\ge^2=0$.  Chevallier~\cite{chev} notes that this should not be regarded as a theorem, as there are exceptions to its application: it is a valuable generic guide. We make use of the principle by starting with invariants of  the rotation group $SO(3)$, acting on $k$ vectors in $\bbR^3$~\cite{weyl}, which we  refer to as $k$-fold invariants.  As was already observed by Study, dualisation of $k$-fold $SO(3)$ invariants leads to dual invariant polynomials, whose real and dual parts are invariants of the Euclidean group.  While the relevant invariants are identified in this manner by Study, he does not explore the question whether they generate all invariants, nor does he provide proofs.

The content of this paper is as follows. In Section~2, the connection between the Euclidean group and the dualised special orthogonal group is set out. Section~3 provides the justification for using dualisation as a means of generating invariants, which does not seem to have been set down previously. The invariants for multiple twists and the associated syzygies are derived in Section~4. Finally, in Section~5, a proof that these invariants rationally generate all invariants in the triple-twist case is given.

\section{The Euclidean group and dualisation}
\label{s:se3}
Let $\bbD$ denote the ring of {\em dual numbers} $a+\ge b$, $a,b\in\bbR$ and $\ge^2=0$ with component-wise addition, and multiplication defined in the obvious way. Note that $\bbD$ is not a field, as there are zero divisors and not every non-zero quantity has a multiplicative inverse. However, it is a 2-dimensional real associative algebra. In the dual number $a+\ge b$, $a$ is referred to as the {\em primal part} and $b$, the {\em dual part}.  Various modules of interest, such as $\Bbb D^3$, can be constructed by taking vectors or matrices of dual numbers; these can also be written as a sum of primal and dual parts. We shall show shortly the relevance of this for the Euclidean group.

The position of a rigid body, for example a component in a robot arm,  with respect to some reference frame is represented by an element of the Euclidean group $SE(3)$.  As noted above, the group is isomorphic to a (semi-direct) product  $SO(3)\ltimes\bbR^3$ of the orientation-preserving rotations $SO(3)$ and  vector translations $\bbR^3$. It is a 6--dimensional Lie group.  
Correspondingly, the Lie algebra $\Lse(3)$ of the Euclidean group is, as a vector space, the direct sum of $\Lso(3)$, the infinitesimal rotations and $\Lt(3)$, infinitesimal translations.  Geometrically, a twist $\bs=(\bgw,\bgv)\in\Lse(3)$ can be interpreted as a vector field on $\bbR^3$ whose integral curves---the motion generated by the twist---are helices of pitch $h=\bgw.\bgv/\bgw.\bgw$ about an axis with direction vector $\bgw$ and moment $\bgw\x\bq=\bgv-h\bgw$ about the origin (where $\bq=(\bgv\x\bgw)/\bgw.\bgw$ is a point on the axis~\cite{selig}). Note that if $h=0$ then the the motion is revolute---the integral curves are circles centred on the axis and lying in planes orthogonal to it. If, on the other hand, $\bgw=\bzero$, then the motion is translational and the integral curves are lines parallel to $\bv$. In this case, the twist is said to have infinite pitch and the corresponding motion is translational. 

The algebra $\Lso(3)$ consists of $3\x3$ skew-symmetric matrices but these in turn can be identified with elements of $\bbR^3$ as in~(\ref{e:skewrep}).  Note that we can also reverse this identification, so that a translation vector $\ba\in\bbR^3$ corresponds to a skew-symmetric matrix $T$, say.  The adjoint action of the Lie group $SE(3)$ on its Lie algebra can then be represented in partitioned matrix form by:
\begin{equation}
\label{adjoint}
(A,\ba).\bs=\begin{bmatrix}A&&\mathrm{O}\\TA&&A\end{bmatrix}\begin{bmatrix}\bgw\\ \bgv\end{bmatrix}
\end{equation}
where $(\bgw,\bgv)$ are the Pl\"ucker coordinates of a twist $\bs\in \Lse(3)$~\cite{selig}. 

The rotation group $SO(3)$ is characterised by the following conditions on a $3\x3$ (real) matrix $A$:
\begin{equation}
\label{e:so3}
AA^t=I,\quad \det A=1.
\end{equation}
Replacing entries in $A$ by dual numbers gives rise to a dual matrix $\hat{A}=A_0+\ge A_1$ where $A_0,A_1$ are real $3\x3$ matrices, the primal and dual parts respectively. The same equations (\ref{e:so3}) determine a group $SO(3,\bbD)$~\cite{mcarthy}. Equating primal and dual parts we obtain:
\begin{equation}
\label{e:so3d}
A_{0}A_{0}^{t}=I,\quad (A_{1}A_{0}^{t})^t=-A_{1}A_{0}^{t}, \quad \det(A_{0})=1,
\end{equation}
so that $A_0\in SO(3)$ and $A_1A_0^t$ is skew symmetric. Identifying this skew-symmetric matrix with a translation vector in $\bbR^3$ gives rise to an isomorphism between $SO(3,\bbD)$ and $SE(3)$ that is at the heart of the principle of transference. Explicitly,
\begin{equation}
\label{dualiso}
\phi:SO(3,\bbD)\to SE(3);\quad \phi(A_0+\ge A_1)=\begin{bmatrix}A_0&\mathrm{O}\\A_1&A_0\end{bmatrix}
\end{equation}
Note that the skew-symmetric matrix in (\ref{adjoint}) $T=A_1A_0^t$.

Dualising works in the Lie algebra, in that  twists in $\Lse(3)$ can be written as dual vectors $\hat{\bgw}=\bgw+\ge\bgv$, where the primal and dual parts $(\bgw,\bgv)$ constitute the Pl\"ucker coordinates.  A change of coordinate frame corresponds to conjugation in the group, giving rise to the Lie bracket in $\Lse(3)$.  Under the isomorphism of $\Lso(3)$ with $\bbR^3$,  the Lie bracket on $\Lso(3)$ corresponds to the standard vector product on $\bbR^3$, $[\bgw_1,\bgw_2]=\bgw_1\x \bgw_2$.
Writing elements of $\Lse(3)$ as dual vectors $\hat{\bgw}_i=\bgw_i+\ge\bgv_i$, $i=1,2$, the Lie bracket is the dual vector product:
\begin{equation}
\label{e:lb}
[\hat{\bgw}_1,\hat{\bgw}_2]=(\bgw_1+\ge\bgv_1)\x(\bgw_2+\ge\bgv_2)=\bgw_1\x\bgw_2+\ge(\bgw_1\x\bgv_2+\bgv_1\x\bgw_2).
\end{equation}
The geometric interpretation of the Lie bracket $[\bs_1,\bs_2]$, in the generic case $\bgw_1\x\bgw_2\neq\bzero$, is as a twist whose axis is the common perpendicular to the axes of $\bs_1,\bs_2$. Its twist $h_{12}$ is given by a function of the pitches $h_i$ of $\bs_i$, $i=1,2$ and the relative placement of the twists~\cite{Samuel}. 

It follows from the isomorphism (\ref{dualiso}) that the adjoint action of $SE(3)$ becomes, in dual form, the standard action of $SO(3,\bbD)$ on $\bbD^3$.  (Note that the adjoint action of $SO(3)$ is equivalent to the standard action under the identification of skew-symmetric matrices and 3-vectors.) This is the essential ingredient in determining invariant polynomials for sets of twists.

\section{Dual mapping and invariants of $SE(3)$}
\label{s:inv}
The principle of transference, as presented by Study in Section~23 of~\cite{study}, is given a modern interpretation by Duffy and Rico~\cite{duffy}, in terms dualisation of the commutative algebra of real differentiable functions. Since we shall mainly be interested in polynomials, it will be sufficient to consider this process on the ring of polynomials in $n$~variables, $\bbR[\gw_1,\dots,\gw_n]=\bbR[\bgw]$.  We regard a polynomial here synonymously with the associated real-valued function. Then the process of dualisation is simply the evaluation of the polynomial in terms of dual variables $\gw_1+\ge\gv_1,\dots,\gw_n+\ge\gv_n$. The resulting function can be identified with a certain type of polynomial in $2n$~variables with dual coefficients.  
Expanding powers and applying $\ge^2=0$ determines the {\bf dual mapping}, $\gD: \bbR[\bgw] \longrightarrow \bbD[\bgw,\bgv] $, defined by:
\begin{align}
\gD(f)(\gw_{1}, \gw_{2},\cdots, \gw_{n})&=\hat{f}(\gw_{1}, \gw_{2},\cdots, \gw_{n};  \gv_{1}, \gv_{2},\cdots,\gv_{n})\notag\\
&=f( \gw_{1}, \gw_{2},\cdots, \gw_{n})+\ge\dsum_{r=1}^{n}  \gv_{r}\displaystyle \frac{\partial f}{\partial  \gw_{r}} (\bgw) \notag\\
&=f(\bgw)+\epsilon \nbb f(\bgw)\bgv.\label{dual}
\end{align}
Note that dualisation is a partial polarisation (see for example~\cite{kraft}).

Given a matrix group $G$ acting linearly on $\bbR^n$, the set of polynomials $f\in\bbR[\bgw]$ invariant under $G$ is denoted $\bbR[\bgw]^G$. A {\bf $k$-fold invariant} of $G$ is a polynomial $f\in\bbR[\bgw_1,\dots,\bgw_k]$ (where each $\bgw_i$ denotes $n$ variables) so that for all $g\in G$,
\begin{equation}
\label{invariant}
f(g.\bgw_1,\cdots,g.\bgw_k)=f(\bgw_1,\cdots,\bgw_k).
\end{equation}
We employ the same terminology for invariants of dual matrix groups acting on the module $\bbD^n$.

\begin{lem}\label{q3}
Let $f\in\bbR[\bgw]^{SO(3)}$. There is a scalar-valued function $\lambda(\bgw)$ such that for $\bgw\in\bbR^3$, the gradient $\nabla f(\bgw)=\lambda(\bgw)\,\bgw^t$.
\end{lem}

\begin{proof}
Given $A\in SO(3)$, by invariance we have $f(A\bgw)=f(\bgw)$.  Given $\bgw'$ so that $\|\bgw'\|=\|\bgw\|$, there exists $A\in SO(3)$ so that $\bgw'=A\bgw$. Hence the level sets of $f$ are (possibly unions of) spheres and the origin $\bzero$. From vector calculus,  $\nabla f(\bgw)\bT=0$ for the tangent $\bT$ to any level curve of $f$ through $\bw$ and hence $\nabla f(\bgw)=\bzero$, or is normal to the sphere through $\bgw$. Since $\bgw$ is also orthogonal to the sphere, therefore there is a scalar $\lambda(\bgw)$ such that $\nabla f(\bgw)=\lambda(\bgw) \bgw^t$.

\end{proof}

The following theorem asserts that dualising a $k$-fold invariant of $SO(3)$ yields a dual invariant polynomial. 

\begin{thm}\label{dualthm1}
Let $f\in\bbR[\bgw_{1},\cdots,\bgw_{k}]^{SO(3)}$, then $\gD(f)=\hat{f}$ is a dual $k$-fold invariant of $SO(3,\bbD)$.
\end{thm}

\begin{proof}
Let $\hat{A}=A_0+\ge A_1\in SO(3,\bbD)$, so that $A_0^tA_0=I$. Furthermore $A^t_{0}A_{1}$ is skew symmetric, therefore  $\bgw^t(A^t_{0}A_{1})\bgw=0$ so that from Lemma~\ref{q3},
\begin{equation}
\label{zeroterm}
\nbb f(\bgw_{1},\cdots,\bgw_{k})(A^t_{0}A_{1}\bgw_{1},\cdots, A^t_{0}A_{1}\bgw_{k})^t=0.
\end{equation}
Since $f$ is invariant,
\begin{equation}
\label{finv}
f(A_{0}\bgw_{1},\cdots,A_{0}\bgw_{k})=f(\bgw_{1},\cdots,\bgw_{k}).
\end{equation}
Suppose $\hat{\v u}_{i}=\bgw_{i}+\ge\bgv_{i}\in\bbD^3$ for  $i=1,\dots,k$. By the chain rule,  
\begin{equation}
\label{chrule}
\nbb f(A_{0}\bgw_{1},\cdots,A_{0}\bgw_{k})=\nbb f(\bgw_{1},\cdots,\bgw_{k})\diag(\underbrace{A^t_{0},\cdots,A^t_{0}}_\text{$k$-times})
\end{equation}
Hence,
\begin{align*}
\hat{f}&\left(\hat {A}\hat{\v u}_{1},\cdots,\hat {A}\hat{\v u}_{k}\right)\\
&=f\left(A_{0}\bgw_1,\cdots,A_{0}\bgw_k;A_{1}\bgw_1+ A_{0}\bgv_1,\cdots,A_{1}\bgw_k+ A_{0}\bgv_k\right)\\
&=f(A_{0}\bgw_{1},\cdots,A_{0}\bgw_{k})+\ge\,\nbb f(A_{0}\bgw_{1},\cdots,A_{0}\bgw_{k})\begin{bmatrix}A_{1}\bgw_{1}+A_{0}\bgv_{1}\\\vdots\\ A_{1}\bgw_{k}+A_{0}\bgv_{k}\end{bmatrix}\;\text{by (\ref{dual})}\\
&=f(\bgw_{1},\cdots,\bgw_{k})\\
&\qquad+\ge\nbb f(\bgw_{1},\cdots,\bgw_{k}) \diag(\underbrace{A^t_{0},\cdots,A^t_{0}}_\text{$k$-times})
\(\begin{bmatrix}A_{0}\bgv_{1}\\\vdots\\ A_{0}\bgv_{k}\end{bmatrix}
+\begin{bmatrix}A_{1}\bgw_{1}\\\vdots\\ A_{1}\bgw_{k}\end{bmatrix}\)\;\text{by (\ref{finv},\ref{chrule})} \\
&=f(\bgw_{1},\cdots,\bgw_{k})+\ge\nbb f(\bgw_{1},\cdots,\bgw_{k}) \(\begin{bmatrix}\bgv_{1}\\\vdots\\ \bgv_{k}\end{bmatrix}+\begin{bmatrix}A^t_{0}A_{1}\bgw_{1}\\\vdots\\ A^t_{0}A_{1}\bgw_{k}\end{bmatrix}\)\\
&=f(\bgw_{1},\cdots,\bgw_{k})+\ge\nbb f(\bgw_{1},\cdots,\bgw_{k}) (\bgv_{1},\cdots, \bgv_{k})^t\quad\text{by (\ref{zeroterm})}\\
 &=f(\bgw_{1},\cdots,\bgw_{k};\bgv_{1},\cdots,\bgv_{k})\\
 &=\hat{f}\left(\hat{\v u}_{1},\cdots,\hat{\v u}_{k}\right).
 \end{align*}
 Therefore, $\hat{f}$ is dual $k$-fold invariant of $SO(3,\bbD)$.
 \end{proof}
 
 This leads us to a method for determining real $k$-fold invariants of the adjoint action of the Euclidean group.
 
\begin{thm}\label{dualthm2}
If $f$ is $k$-fold invariant of the adjoint action of $SO(3)$, then the primal and dual parts of $\gD(f)=\hat{f}$ are (real) $k$-fold invariants of the adjoint action of $SE(3)$. Furthermore, the primal and dual parts of the dualisation of any syzygy among $SO(3)$ $k$-fold invariants are syzygies for $SE(3)$ invariants.
\end{thm}
 
\begin{proof}
The first part follows immediately from Theorem~\ref{dualthm1} and the identification of the adjoint action of $SE(3)$ on $\Lse(3)$ with the action of $SO(3,\bbD)$ on $\bbD^3$. That syzygies arise from dualisation is a consequence of the observation of Rico Mart\'{i}nez and Duffy~\cite{duffy} that the dualisation $\gD$ is an algebra homomorphism. 
\end{proof}

\section{Invariants and dualisation}
The standard action of the rotation group $SO(3)$ on $\bbR^3$ is the same as its adjoint action on $\Lso(3)$ under the identification of 3-vectors with skew-symmetric matrices~(\ref{e:skewrep}).  Weyl~\cite{weyl} gives a complete account of the vector  invariants for the standard action of the orthogonal groups of all orders, which for $n=3$ we may therefore treat as vector invariants for the adjoint action.  There are two basic types: given $\bgw_1,\dots,\bgw_m\in\bbR^3\cong\Lso(3)$, every $m$-fold polynomial invariant is generated by (i.e. is a polynomial function of) the quadratic and cubic  invariants 
\begin{align}
I_{ij}&=\bgw_i\cdot\bgw_j,  \quad 1\leq i,j\leq m\notag\\
I_{ijk}&=[\bgw_i\;\bgw_j\;\bgw_k],\quad 1\leq i<j<k\leq m\label{e:finv2}
\end{align}
where the {\em bracket} in $I_{ijk}$ denotes the determinant of the matrix whose columns are the three vectors. Note that these invariants are not algebraically independent but are linked by several types of {\em syzygies}, that is, polynomial relations. We shall only be concerned with $m\leq3$ and in these cases the only syzygy occurs when $m=3$. There are six quadratic invariants and one of cubic type connected by the single syzygy
\begin{equation}
\label{e:syz}
 I_{123}^2=\det(I_{ij}).
 \end{equation}
  
For the adjoint action itself, $SO(3)$ has the single generating invariant $I_{11}=\bgw\cdotp\bgw$ which dualises to give:
\begin{equation}
\label{e:dualinv}
(\bgw+\ge\bgv)\cdotp(\bgw+\ge\bgv)=\bgw\cdotp\bgw+2\,\ge\,\bgw\cdotp\bgv.
\end{equation}
By Theorem~\ref{dualthm2}, the primal and dual parts are invariants $I_{11}$ and $\tilde{I}_{11}=2\bgw.\bv$, respectively.  Up to a constant multiple, these are the familiar Killing and Klein forms whose ratio is the pitch of the twist $S=(\bgw,\bgv)$. They are known to generate the ring of polynomial invariants. 

For $m=2$, there are 6~quadratic invariants arising from dualisation of the $SO(3)$ invariants (see~\cite{crook} for an alternative derivation of these invariants).  In the case $m=3$, there are 14~invariants arising from the primal and dual parts of the dualisations of the 7~generating $3$-fold invariants for $SO(3)$:
\begin{equation}
\label{e:3inv}
\begin{array}{lll}
I_{ij}=\bgw_{i}\cdotp\bgw_{j},&\tilde{I}_{ij}=\bgw_{i}\cdotp\bgv_{j}+\bgv_{i}\cdotp\bgw_{j},\quad&1\leq i\leq j\leq3\\
I_{123}=[\bgw_{1}\;\bgw_{2}\;\bgw_{3}],\quad&
\multicolumn{2}{l}{\tilde{I}_{123}=\dsum_{\sigma\in C_{3}}[\bgw_{\sigma(1)}\;\bgw_{\sigma(2)}\;\bgv_{\sigma(3)}].}
\end{array}
\end{equation}
where $C_{3}$ denotes the cyclic group of order $3$. Note that $\tilde{I}_{ii}=2\bgw_{i}\cdotp\bgv_{i}$ for $i=1,2,3$. 

The first $12$ quadratic $3$-fold invariants  of the adjoint action of $SE(3)$ in (\ref{e:3inv})  are algebraically independent, since their Jacobian matrix has rank~$12$~\cite{daher}.  The syzygy~(\ref{e:syz}) dualises to give, on the left-hand side:
\begin{align*}
[\bgw_{1}+\epsilon\bgv_{1}&\quad\bgw_{2}+\epsilon\bgv_{2}\quad\bgw_{3}+\epsilon\bgv_{3}]^2\\  
&=\(|\bgw_{1}\;\bgw_{2}\;\bgw_{3}|+|\epsilon\bgv_{1}\;\bgw_{2}\;\bgw_{3}|+|\bgw_{1}\;\epsilon\bgv_{2}\;\bgw_{3}|+|\bgw_{1}\;\bgw_{2}\;\epsilon\bgv_{3}|\)^2\\
&=\(|\bgw_{1}\;\bgw_{2}\;\bgw_{3}|+\epsilon(|\bgv_{1}\;\bgw_{2}\;\bgw_{3}|+|\bgw_{1}\;\bgv_{2}\;\bgw_{3}|+|\bgw_{1}\;\bgw_{2}\;\bgv_{3}|)\)^2\\
&=\(|\bgw_{1}\;\bgw_{2}\;\bgw_{3}|\)^2+2\,\epsilon\,|\bgw_{1}\;\bgw_{2}\;\bgw_{3}|\(|\bgv_{1}\;\bgw_{2}\;\bgw_{3}|+|\bgw_{1}\;\bgv_{2}\bgw_{3}|+|\bgw_{1}\;\bgw_{2}\;\bgv_{3}|\)\\
&=I_{123}^2+2\,\epsilon\, I_{123}\tilde{I}_{123},
 \end{align*}
 while the right-hand side gives:
 \begin{align*}
 &\left|\begin{matrix} I_{11}+\epsilon\tilde{I}_{11}\;& I_{12}+\epsilon\tilde{I}_{12}\;&I_{13}+\epsilon\tilde{I}_{13}\\
  I_{12}+\epsilon\tilde{I}_{12}&I_{22}+\epsilon\tilde{I}_{22}&I_{23}+\epsilon\tilde{I}_{23}\\
  I_{13}+\epsilon\tilde{I}_{13}&I_{23}+\epsilon\tilde{I}_{23}&I_{33}+\epsilon\tilde{I}_{33}\\
              \end{matrix}\right|\\
               &\qquad=\left|\begin{matrix} I_{11}& I_{12}&I_{13}\\
  I_{12}&I_{22}&I_{23}\\
  I_{13}&I_{23}&I_{33}\\
              \end{matrix}\right|+
              \epsilon\(\left|\begin{matrix} \tilde{I}_{11}& I_{12}&I_{13}\\
 \tilde{I}_{12}&I_{22}&I_{23}\\
  \tilde{I}_{13}&I_{23}&I_{33}\\
              \end{matrix}\right|+\left|\begin{matrix} I_{11}& \tilde{I}_{12}&I_{13}\\
  I_{12}&\tilde{I}_{22}&I_{23}\\
  I_{13}&\tilde{I}_{23}&I_{33}\\
              \end{matrix}\right|+\left|\begin{matrix} I_{11}& I_{12}&\tilde{I}_{13}\\
  I_{12}&I_{22}&\tilde{I}_{23}\\
  I_{13}&I_{23}&\tilde{I}_{33}\\
              \end{matrix}\right|\)\\
&\qquad=(I_{11}I_{22}I_{33}-I_{11}I_{23}^2+2I_{12}I_{13}I_{23}-I_{12}^2I_{33}-I_{13}^2I_{22})\\
&\qquad\qquad+\epsilon\,(\tilde{I}_{11} I_{22} I_{33}+ I_{11}\tilde{I}_{22} I_{33} +I_{11} I_{22}\tilde{I}_{33}\\
&\qquad\qquad\qquad-\tilde{I}_{11}I_{23}^2- 2I_{11} I_{23}\tilde{I}_{23}+ 2\tilde{I}_{12}I_{13} I_{23}+ 2I_{12}\tilde{I}_{13}I_{23}+ 2I_{12} I_{13}\tilde{I}_{23}\\
&\qquad\qquad\qquad- 2I_{12}\tilde{I}_{12} I_{33}- I_{12}^2\tilde{I}_{33}-2I_{13}\tilde{I}_{13} I_{22}- I_{13}^2\tilde{I}_{22}).
 \end{align*}
 The terms of the dual part are ordered to emphasise the differential nature of dualisation.
 
 Equating primal and dual parts of the two expressions we obtain the pair of syzygies: 
 \begin{align}
I_{123}^2&= I_{11}I_{22}I_{33}-I_{11}I_{23}^2+2I_{12}I_{13}I_{23}-I_{12}^2I_{33}-I_{13}^2I_{22}\label{syz1}\\
2I_{123}\tilde{I}_{123}&=\tilde{I}_{11} I_{22} I_{33}+ I_{11}\tilde{I}_{22} I_{33} +I_{11} I_{22}\tilde{I}_{33}\notag\\
&\qquad-\tilde{I}_{11}I_{23}^2- 2I_{11} I_{23}\tilde{I}_{23}+2 \tilde{I}_{12}I_{13} I_{23}+ 2I_{12}\tilde{I}_{13}I_{23}+ 2I_{12} I_{13}\tilde{I}_{23}\notag\\
&\qquad-2 I_{12}\tilde{I}_{12} I_{33}- I_{12}^2\tilde{I}_{33}-2I_{13}\tilde{I}_{13} I_{22}- I_{13}^2\tilde{I}_{22}.\label{syz2}
  \end{align}
 Bracket versions of these syzygies appear in the work of Study~\cite{study}, Section~23.

 \section[Rationally Invariant Function]{Finite rational generation of $3$-fold invariants}

Having obtained finite lists of $3$-fold invariants and syzygies for $SE(3)$, one would like to establish fundamental theorems asserting that these generate the corresponding rings. The fact that $SE(3)$ fails to be reductive means that such theorems could not follow from the classical theory.  Panyushev~\cite{pchev} uses a theorem of Igusa to establish a general result for so-called Takiff Lie groups $G\ltimes\Lg$ (of which $SE(3)$ is an example). This asserts that, in the absence of syzygies, a fundamental set of $G$--adjoint invariants  and their duals generate the $G\ltimes\Lg$ adjoint invariants. This suffices for the $k$-fold invariants of $SE(3)$, $k=2$, where there are no syzygies, but not in the case $k=3$.    In this section we follow Weyl's method~\cite{weyl} to show  that  all rational $3$-fold invariants of the adjoint action of $SE(3)$  can at least be expressed as rational functions of the $14$~invariants in (\ref{e:3inv}).  Indeed, $13$ suffice, since (\ref{syz2}) enables us to express $\tilde{I}_{123}$ rationally in terms of the remaining invariants.
 
We modify Weyl's argument for $SO(n)$ to the semi-direct product  $SE(3)= SO(3)\ltimes\bbR^3$. In particular, we introduce even and odd invariants of the adjoint action of $E(3)= O(3)\ltimes\bbR^3$ and establish that these are all invariants of the adjoint action of $SE(3)$. 
 
 \subsection{Even and odd invariants}
Let $O(3)=SO(3)\cup O^{-}(3)$ so that for $R\in  O^-(3)$, $\det R=-1$. We also denote by $E^-(3)$ the Cartesian product $O^-(3)\x\bbR^3$ and recall the identification (\ref{e:skewrep}) of $\Lso(3)$ and $\bbR^3$ which is used throughout the following. 

 \begin{definition} A polynomial $f\in\bbR[\bgw_1,\bgv_1,\cdots,\bgw_k,\bgv_k]$ under the $k$-fold  adjoint action of the isometry group $E(3)$ is called:
 \begin{enumerate}
   \item an \textbf{even invariant} if  for all $(R,T)\in E(3)$:
 \[
 f(R\bgw_1,\cdots,TR\bgw_k+R\bgv_k,\cdots,R\bgw_k,TR\bgw_k+R\bgv_k)=f(\bgw_1,\bgv_1,\cdots,\bgw_k,\bgv_k);
 \]
     \item an \textbf{odd invariant} if  for all $(R,T)\in E(3)$:
 \[
   f(R\bgw_1,\cdots,TR\bgw_k+R\bgv_k,\cdots,R\bgw_k,TR\bgw_k+R\bgv_k)=\det(R)\cdot f(\bgw_1,\bgv_1,\cdots,\bgw_k,\bgv_k),
   \]
 \end{enumerate}
 \end{definition}
 
Denote by:
\begin{align*}
&A_{_{E}}^{^{(k)}}\quad\text{the subalgebra of all even $k$-fold invariants of the adjoint action of $E(3)$,}\\
&A_{_{O}}^{^{(k)}}\quad\text{the subalgebra of all odd  $k$-fold invariants of the adjoint action of $E(3)$,}\\
&B^{^{(k)}}\quad\text{the subalgebra of all $k$-fold invariants of the adjoint action of $SE(3)$.}
\end{align*}
 
 \begin{thm}\label{oo}
  For all $k$, $A_{_E}^{^{(k)}}+A_{_O}^{^{(k)}}=B^{^{(k)}}$.
 \end{thm}
 
 \begin{proof}
  Let us fix $k$ and suppress the superscript $(k)$. Since $SO(3)$ is a subgroup of $O(3)$ and for $R\in SO(3)$, $\det R=1$, we have $A_{_E}, A_{_O}\subseteq B$, so that $A_{_E}+A_{_O}\subseteq B$.
  
  To prove the reverse inclusion, let us assume $f\in B$, so $f$ is invariant under the $k$-fold adjoint action of $SE(3)$. Given  $(R, T)\in E^-(3)$, define 
 \begin{align}
  \label{fdash}
  f'_{(R,T)}(\bgw_1,\cdots,\bgv_k)=f(R\bgw_1,\cdots,TR\bgw_k+R\bgv_k).
\end{align} 
 We shall prove that for any $R_1,R_2\in O^-(3)$, $f'_{(R_1,T_1)}=f'_{(R_2,T_2)}$.  As a result, the products $R_1R_2, R_2^2 \in SO(3)$, and therefore $(R_1R_2,T_1),(R_2^2,T_2)\in SE(3)$. Given that $f$ is $SE(3)$--invariant:
  \begin{align*}
   f'_{(R_1,T_1)}(\bgw_1,\cdots\!,\bgv_k)&=f(R_1\bgw_1,\cdots,T_1R_1\bgw_k+R_1\bgv_k)\\
   &=f(R_1(R_2R^{-1}_2\bgw_1),\cdots,T_1R_1(R_2R^{-1}_2\bgw_k)+R_1(R_2R^{-1}_2\bgv_k))\\
   &=f(R_1R_2(R^{-1}_2\bgw_1),\cdots,T_1R_1R_2(R^{-1}_2\bgw_k)+R_1R_2(R^{-1}_2\bgv_k))\\
   &=f(R^{-1}_2\bgw_1,\cdots,R^{-1}_2\bgv_k)\\
   &=f(R^2_2(R^{-1}_2\bgw_1),\cdots,T_2R^2_2(R^{-1}_2\bgw_k)+R^2_2(R^{-1}_2\bgv_k))\\
   &=f(R_2\bgw_1,\cdots,T_2R_2\bgw_k+R_2\bgv_k)\\
   &=f'_{(R_2,T_2)}(\bgw_1,\cdots,\bgv_k).
  \end{align*}
 So for any $f\in B$, let $f'$ to be the {\em unique} element of the set 
 \[
 \{f'_{(R,T)}\,|\,(R,T)\in E^-(3)\}.
 \]
 Define
 \begin{align}
 \label{odev} 
 f_{_{E}}=\frac{1}{2}(f+f'),\quad \text{and}\quad f_{_O}=\frac{1}{2}(f-f').
 \end{align}
 Clearly $f=f_{_E}+f_{_O}$, and it remains to show that $f_{_E}\in A_{_E}$ and $f_{_O}\in A_{_O}$.
 To show $f_{_E}\in A_{_E}$, let  $(R,T)\in E(3)$, then  the  two  cases, $R\in SO(3)$ and $R\in O^-(3)$, will be considered separately. 
 
 Suppose $R\in SO(3)$, then clearly $-R\in O^-(3)$.  For any $(R^*, T^*)\in E^-(3)$ we have, using the invariance of $f$ and the uniqueness of $f'$:
 \begin{align*}
  f_{_E}(R\bgw_1,&\cdots,TR\bgw_k+R\bgv_k)\\
  &=\frac{1}{2}[f(R\bgw_1,\cdots,TR\bgw_k+R\bgv_k)+f'(R\bgw_1,\cdots,TR\bgw_k+R\bgv_k)]\\
  &=\frac{1}{2}[f(\bgw_1,\cdots,\bgv_k)+f(R^*(R\bgw_1),\cdots,T^*R^*(R\bgw_k)+R^*(TR\bgw_k+R\bgv_k))],
  \intertext{and choosing, in particular, $R^*=-I_3$ and $T^*=O_{_{3\times3}}$}
  &=\frac{1}{2}[f(\bgw_1,\bgv_1)+f(-R\bgw_1,T(-R)\bgw_1+(-R)\bgv_1)]\\
  &= \frac{1}{2}[f(\bgw_1,\cdots,\bgv_k)+f'(\bgw_1,\cdots,\bgv_k)]\\
  &= f_{_E}(\bgw_1,\cdots,\bgv_k).
 \end{align*}
 On the other hand, for $R\in O^-(3)$, using a similar argument but with $-R\in SO(3)$, we have:
 \begin{align*}
  f_{_E}(R\bgw_1,&\cdots,TR\bgw_k+R\bgv_k)\\
  &=\frac{1}{2}[f(R\bgw_1,\cdots,TR\bgw_k+R\bgv_k)+f'(R\bgw_1,\cdots,TR\bgw_k+R\bgv_k)]\\
  &=\frac{1}{2}[f'(\bgw_1,\cdots,\bgv_k)+f(-R\bgw_1,\cdots,T(-R)\bgw_k+(-R)\bgv_k)]\\
  &= \frac{1}{2}[f'(\bgw_1,\cdots,\bgv_k)+f(\bgw_1,\cdots,\bgv_k)]\\
  &= f_{_E}(\bgw_1,\cdots,\bgv_k).
  \end{align*}
  Therefore $f_{_E}\in A_{_E}$. The argument that $f_{_O}\in A_{_O}$ is essentially the same and the theorem follows. 
 \end{proof}
 
 \subsection{A normal form for three twists}
 In order to express $3$--fold invariants of the adjoint action of $SE(3)$ in a standard form,  we make a change of coordinates in the space of twists. We transform three given non-zero twists $\bs_i=(\bgw_i,\bgv_i)$ to a normal form $S'_i=(\bgw'_i,\bgv'_i)$, for $i=1,2,3$ using the Euclidean transformation: 
 \begin{align*}
  \bgw_i\longmapsto \bgw_i'=R\bgw_i,  \qquad\qquad\bgv_i\longmapsto \bgv_i'=TR\bgw_i+R\bgv_i.
 \end{align*}

First, choose a rotation matrix  $R\in SO(3)$ so that $\bgw_1$  along the $x_1$--axis,   $\bgw_2$ lies in the $x_1x_2$--plane, and $\bgw_3$ transforms to an arbitrary vector in the space, that is:
 \begin{align}\label{ddd1}
 R\bgw_1=\begin{bmatrix}\alpha_1\\0\\0\end{bmatrix},\qquad R\bgw_2=\begin{bmatrix}\alpha_2\\\alpha_3\\0\end{bmatrix},\qquad R\bgw_3=\begin{bmatrix}\alpha_4\\ \alpha_5\\ \alpha_6\end{bmatrix},
 \end{align}
 for some $\alpha_i$, $i=1,\cdots,6$.  Note that this can be done even if $\bgw_i=\b0$ for any $i=1,2,3$. However, assume for the moment that $\bgw_1,\bgw_2$ are linearly independent so that $\ga_1,\ga_3\neq0$. 
 
Now, given $T$ skew-symmetric, of the form:
\[
T=\begin{bmatrix}
0&-t_3&t_2\\t_3&0&-t_1\\-t_2&t_1&0
\end{bmatrix}
\]
we obtain: 
\begin{align}
 \label{ddd2}
 \bgv_1' =(v_{11},\; -t_3\alpha_1+v_{12},\; t_2\alpha_1+v_{13})^t,
 \end{align}
 so that, with $\alpha_1\neq0$,  setting  $t_2=-v_{13}/\alpha_1$ and $t_3=-v_{12}/\alpha_1$ gives $v_{12}'=v_{13}'=0$. Likewise,  we have: 
 \begin{align}
\bgv_2'=(t_3\alpha_3+v_{21},\; t_3\alpha_2+v_{22},\; -t_2\alpha_2+t_1\alpha_3+v_{23})^t.
 \end{align}
Then, with $\ga_3\neq0$,  $t_1=(t_2\alpha_2-v_{23})/\alpha_3$ gives $v_{23}'=0$. 

If, in fact $\bgw_1=\bzero$ then $\bgv_1\neq\bzero$ and we choose $R$ so that
\[
R\bgv_i=\begin{bmatrix}\beta_1\\0\\0\end{bmatrix}.
\]
Likewise, if $\bgw_1,\bgw_2$ are linearly dependent, then $R$ can be chosen so that $\bv_2$ lies on the $x_1x_2$--plane in 3--space. In all cases, we therefore obtain a change of coordinates so that the twists have the form  $\bs'_i=(\bgw'_i,\bgv'_i)$ $i=1,2,3$,  where:
 \begin{align}
 \label{basis1}
  \bs'_1=\begin{bmatrix}
        \alpha_1\\
        0\\
        0\\
        \beta_1\\
        0\\
        0\\
       \end{bmatrix},\qquad\qquad \bs'_2=\begin{bmatrix}
        \alpha_2\\
        \alpha_3\\
        0\\
        \beta_2\\
        \beta_3\\
        0\\
       \end{bmatrix},\qquad \qquad \bs'_3=\begin{bmatrix}
        \alpha_4\\
        \alpha_5\\
        \alpha_6\\
        \beta_4\\
        \beta_5\\
        \beta_6\\
       \end{bmatrix},
 \end{align}
 and $\alpha_i,\beta_i\in\bbR$  for  $i=1,\cdots,6$.
 
 \subsection{Even invariants  of  $E(3)$}
  
  We  may now make use of the normal form to establish finite generation theorems for $E(3)$ invariants. First note that the quadratic invariants $I_{ij}$, $\tilde{I}_{ij}$, $1\leq i\leq j\leq 3$ are even since for {\em any} orthogonal $R\in O(3)$ and $\bx,\by\in\bbR^3$, we have $R\bx\cdot R\by=\bx.\by$. 

 \begin{thm}
 \label{eventhm}
Every even $3$-fold polynomial invariant $f$ of the adjoint action of $E(3)$  can be expressed as a polynomial function of the 12~quadratic vector invariants $I_{ij}, \tilde{I}_{ij}$, $1\leq i\leq j\leq3$. 
 \end{thm}
 
 \begin{proof}
Given a set of three twists  $\bs_i=(\bgw_i,\bgv_i)$, $i=1,2,3$ let $\bs'_i=(\bgw'_i,\bgv'_i)$, $i=1,2,3$ denote their normal forms as in (\ref{basis1}). Evaluating the 6~primal quadratic invariants on $\bs'_i$, we have:
 \begin{align}
  &I_{11}=\alpha_1^2\quad
  &&I_{12}=\alpha_1\alpha_2\notag\\
  &I_{22}=\alpha_2^2+\alpha_3^2\quad
  &&I_{13}=\alpha_1\alpha_4\notag\\
  &I_{33}=\alpha_4^2+\alpha_5^2+\alpha_6^2
  &&I_{23}=\alpha_2\alpha_4+\alpha_3\alpha_5.\label{inveq1}
\intertext{while the dual-part quadratic invariants are:}
 & \tilde{I}_{11}=2\alpha_1\beta_1
 &&\tilde{I}_{12}=\alpha_1\beta_2+\beta_1\alpha_2\notag\\ 
 &\tilde{I}_{22}=2\alpha_2\beta_2+2\alpha_3\beta_3  
 && \tilde{I}_{13}=\alpha_1\beta_4+\beta_1\alpha_4\notag\\ 
 &\tilde{I}_{33}=2\alpha_4\beta_4+2\alpha_5\beta_5+2\alpha_6\beta_6
 && \tilde{I}_{23}=\alpha_2\beta_4+\alpha_3\beta_5+\alpha_4\beta_2+\alpha_5\beta_3.\label{inveq2}
 \end{align}
 
However, even invariants are also fixed by  improper rotations, so we can additionally make any of the following three (simultaneous) transformations by changing the sign of one coordinate:
 \begin{align}
 \label{transs}
&\alpha_6 \longrightarrow -\alpha_6\quad \text{and}\quad \beta_6 \longrightarrow -\beta_6\notag\\
&\alpha_3,\alpha_5\longrightarrow-\alpha_3,-\alpha_5\quad \text{and}\quad\beta_3,\beta_5\longrightarrow-\beta_3,-\beta_5
\notag\\
&\alpha_1,\alpha_2,\alpha_4\longrightarrow-\alpha_1,-\alpha_2,-\alpha_4\quad \text{and}\quad\beta_1,\beta_2,\beta_4\longrightarrow-\beta_1,-\beta_2,-\beta_4
 \end{align}
 The polynomial $f$ is a linear combination of monomials
 \begin{equation}
 \label{monomial}
M=\alpha_1^{a_1}\alpha_2^{a_2}\alpha_3^{a_3}\alpha_4^{a_4}\alpha_5^{a_5}\alpha_6^{a_6}\beta_1^{b_1}\beta_2^{b_2}\beta_3^{b_3}\beta_4^{b_4}\beta_5^{b_5}\beta_6^{b_6}.
 \end{equation}
Invariance under the three transformations (\ref{transs})  gives:
\begin{itemize}
\item
the  exponents $a_6$ and $b_6$ are of equal parity, that is, either both even, or both odd;
\item
the sum of the exponents $a_3$, $a_5$, $b_3$, and $b_5$ must be even;
\item
the sum of the exponents $a_1$, $a_2$, $a_4$, $b_1$, $b_2$, and $b_4$ must be even.
\end{itemize}
It follows that $M$ can be written as a product of powers of the following quadratic monomials:
\begin{align}
\label{qmons}
&\alpha_6^2, \beta_6^2, \alpha_6\beta_6\notag\\
&\alpha_3^2, \alpha_5^2, \alpha_3\alpha_5, \beta_3^2, \beta_5^2, \beta_3\beta_5, \alpha_5\beta_5, \alpha_3\beta_3, \alpha_3\beta_5, \alpha_5\beta_3\notag\\
&\alpha_1^2, \alpha_2^2, \alpha_4^2, \beta_1^2, \beta_2^2, \beta_4^2, \alpha_1\alpha_2, \alpha_1\alpha_4, \alpha_2\alpha_4, \beta_1\beta_2,  \beta_1\beta_4,  \beta_2\beta_4,\notag\\
&\kern2cm\alpha_1\beta_1, \alpha_1\beta_2, \alpha_1\beta_4, \alpha_2\beta_1, \alpha_2\beta_2, \alpha_2\beta_4, \alpha_4\beta_1, \alpha_4\beta_2, \alpha_4\beta_4.
\end{align}
Equations (\ref{inveq1}),  (\ref{inveq2}) can be solved simultaneously to provide rational expressions for each of these basic monomials  in terms of the invariants.  For example, we have immediately $ \alpha_1^2=I_{11}$ and $\ga_1\ga_2=I_{12}$, from which we deduce $\ga_1^2\ga_2^2=I_{12}^2$, so that $\ga_2^2=I_{12}^2/I_{11}$. Likewise, we have $\ga_1\gb_1=\frac12\tilde{I}_{11}$ and hence:
\begin{equation}
\label{invsolve}
\ga_2\gb_1=\frac{(\ga_i\ga_2)(\ga_1\gb_2)}{\ga_1^2}=\frac{I_{12}\tilde{I}_{11}}{2I_{11}}.
\end{equation}
The expressions for $\ga_6^2$  becomes: 
\begin{align}
\label{uu1}
\alpha_6^2&=\frac{I_{11}I_{22}I_{33}-I_{11}I_{23}^2+2I_{12}I_{13}I_{23}-I_{12}^2I_{33}-I_{13}^2I_{22}}{I_{11}I_{22}-I_{12}^2}=\frac{I_{123}^2}{I_{11}I_{22}-I_{12}^2}
\end{align}
from which we can obtain by applying a differentiation rule:
\begin{equation}
2\ga_6\gb_6=\frac{2I_{123}\tilde{I}_{123}(I_{11}I_{22}-I_{12}^2)-I_{123}^2(\tilde{I}_{11}I_{22}+I_{11}\tilde{I}_{22}-2I_{12}\tilde{I}_{12})}{(I_{11}I_{22}-I_{12}^2)^2}.
 \end{equation}
The monomial $\ga_6\beta_6$ can then be rewritten in rational form in terms of the quadratic invariants by reversing these of the syzygies (\ref{syz1},\ref{syz2}).  
 
 In summary, $f$ can be written in the form:
 \begin{equation}
 \label{ratinv}
 \frac{p(I_{11},\cdots,\tilde{I}_{23})}{I_{11}^{r}(I_{11}I_{22}-I_{12}^2)^{s}},
 \end{equation} 
for some polynomial $p$ and non-negative integers $r,s$. Observe, however, that we could have permuted the order in which the normal form  of the original three twists was derived and thereby obtained an alternative form for $f$, in which the subscripts would be permuted likewise. From this it follows that, for example:
 \[
 p(I_{11},\cdots,\tilde{I}_{23})I_{22}^{r'}(I_{22}I_{33}-I_{23}^2)^{s'}-q(I_{11},\cdots,\tilde{I}_{23})I_{11}^{r}(I_{11}I_{22}-I_{12}^2)^{s}\equiv0,
 \]
 for a polynomial $q$ and non-negative integers $r'$, $s'$. Since the 12~quadratic invariants are algebraically independent, the polynomial 
 \[
 p(\xi_1,\cdots,\xi_6,\eta_1,\cdots,\eta_6)\xi_2^{r'}(\xi_2\xi_3-\xi_6^2)^{s'}-q(\xi_1,\cdots,\xi_6,\eta_1,\cdots,\eta_6)\xi_1^{r}(\xi_1\xi_2-\xi_4^2)^{s}\equiv0
 \]
 so that $\xi_1^{r}$, $(\xi_1\xi_2-\xi_4^2)^{s}$ divide $p$ and we may reduce the expression (\ref{ratinv}) to a polynomial in $I_{11},\cdots, \tilde{I}_{23}$. 
\end{proof}

 \subsection{Odd invariants  of  $E(3)$}
The cubic invariants are themselves odd since for any $R\in O(3)$ and $3\x3$ matrix $W$, $\det(RW)=\det R.\det W$.  Evaluating the cubic invariants on the standard form for three twists gives:
 \begin{align}
 I_{123}&=\left|\begin{matrix}\alpha_1&\alpha_2&\alpha_4\\0&\alpha_3&\alpha_5\\0&0&\alpha_6 \end{matrix}\right|=\alpha_1\alpha_3\alpha_6\label{I13}\\
  \tilde{I}_{123}&=\left|\begin{matrix}\beta_1&\alpha_2&\alpha_4\\0&\alpha_3&\alpha_5\\0&0&\alpha_6 \end{matrix}\right|+\left|\begin{matrix}\alpha_1&\beta_2&\alpha_4\\0&\beta_3&\alpha_5\\0&0&\alpha_6 \end{matrix}\right|+\left|\begin{matrix}\alpha_1&\alpha_2&\beta_4\\0&\alpha_3&\beta_5\\0&0&\beta_6 \end{matrix}\right|=\beta_1\alpha_3\alpha_6+\alpha_1\beta_3\alpha_6+\alpha_1\alpha_3\beta_6.\label{I14}
 \end{align}

 \begin{thm}\label{oddthm}
  The odd  invariants of 3-fold  adjoint action of $E(3)$ can be expressed as rational functions of the 12~quadratic vector invariants (\ref{e:3inv}) and  $I_{123}$.
 \end{thm}

 \begin{proof}
Any odd invariant $f$ of the action of $SE(3)$ on a set of three twists must change sign under improper rotations and thus under the transformations itemised in (\ref{transs}). It follows that for every monomial $M$ as in (\ref{monomial}), in the expression for $f$ evaluated on a normal form (\ref{basis1}) the sum of the exponents of the variables in each of the following sets $\Theta_i$, $i=1,2,3$, must be  odd:
 \begin{align*}
  \Theta_1&=\{\alpha_1,\alpha_2,\alpha_4,\beta_1,\beta_2,\beta_4\}\\
  \Theta_2&=\{\alpha_3,\alpha_5, \beta_3,\beta_5\}\\
  \Theta_3&=\{\alpha_6, \beta_6\}.
 \end{align*}
For example, considering only the variables in $\Theta_1$, we have: the two possibilities:
  \begin{align}
   \alpha_6^{2n+1}\beta_6^{2m}&=(\alpha_6^{2})^{n}(\beta_6^{2})^{m}\alpha_6\notag\\
   \alpha_6^{2n}\beta_6^{2m+1}&=(\alpha_6^{2})^{n}(\beta_6^{2})^{m}\beta_6,
  \end{align}
and similarly for the terms in $\Theta_2,\Theta_3$.  Therefore $M=\th_1\th_2\th_3\,M'$, where $M'$ is  a monomial of even degree that can be written in terms of the monomials listed in (\ref{qmons}) and $\th_i\in \Theta_i$, $i=1,2,3$.  Every such cubic monomial can be rewritten as follows:
\begin{equation}
\th_1\th_2\th_3=\frac{(\ga_1\th_1)(\ga_3\th_2)(\ga_6\th_3)\ga_1\ga_3\ga_6}{\ga_1^2\ga_3^2\ga_6^2}=\frac{p(I_{11},\cdots,\tilde{I}_{23})}{q(I_{11},\cdots,I_{23})}I_{123}
\end{equation}
where $p$ is the polynomial determined by the expressions for quadric invariants obtained in Theorem~\ref{eventhm}, $q$ is the right-hand side in (\ref{syz1}) and we use equation (\ref{I13}).

It follows immediately that every odd 3-fold invariant of the adjoint action of $SE(3)$ is rationally generated by the 12 quadric invariants together with $I_{123}$.  Note that a similar trick would enable us to replace $I_{123}$ by $\tilde{I}_{123}$.
 \end{proof}
 
 Combining the results for even and odd invariants gives the following.

 \begin{cor}\label{ratgen}
Every invariant of the 3-fold  adjoint action of $SE(3)$ can be expressed as  a rational function of the 12~quadratic  invariants (\ref{e:3inv}) and  $I_{123}$.
 \end{cor}
 
 \section{Conclusion}
 It would clearly be desirable to go beyond the Corollary~\ref{ratgen} to obtain a polynomial finite-generation theorem for multiple twists.  However, both the original argument of Weyl for vector invariants of the special orthogonal group and the method employed by Panyushev for the Euclidean group and two twists fail for $k\geq3$ twists because of the existence of syzygies among the known invariants.  Nevertheless, we conjecture that the 14~invariants listed in (\ref{e:3inv}) do generate the ring of polynomial invariants in this case.

\section*{Acknowledgements}
The authors warmly acknowledge many invaluable conversations with our colleague Dr Petros Hadjicostas.


\begin{thebibliography}{99}

\bibitem{chev} Chevalier, D.P.: On the transference principle in kinematics: its various forms and limitations. Mech. mach. Theory {\bf31}, 57--76 (1995)

\bibitem{crook} Crook, D.: Polynomial invariants of the {E}uclidean group action on multiple screws. Master's thesis, Victoria University of Wellington (2009)

\bibitem{daher}Daher, M.: Dual numbers and the invariant theory of the {E}uclidean group with applications to robotics. Ph.D. thesis, Victoria University of Wellington (2013)

\bibitem{dahdon}Daher, M., Donelan,~P.S.: Invariant properties of the {D}enavit--{H}artenberg parameters. In A. Kecskemethy, F.G. Flores (eds.) Interdisciplinary Applications of Kinematics, Mechanism and Machine Science vol. 26, pp. 43--51. Springer, Cham Switzerland (2015)

\bibitem{gibson} Donelan,~P.S.,
Gibson,~C.G.: First--order  invariants of {E}uclidean motions.  Acta Appl. Math. {\bf 24}, 233--251 (1991) 

\bibitem{hunt}  Gibson,~C.G., Hunt,~K.H.:
Geometry of screw  systems I \& II.  Mech. Mach. Theory {\bf 25}, 1--27 (1990)  

\bibitem{kraft} Kraft, H. Procesi, C.: Cassical invariant theory, a primer. \url{http://www.math.unibas.ch/~kraft/Papers/KP-Primer.pdf} (1996) Accessed 2014-11-09

\bibitem{mcarthy} McCarthy, J.: Dual orthogonal matrices in manipulator kinematics. Int. J. Robotics Res. {\bf 5}, 45--51 (1986)

\bibitem{pchev} Panyushev, D.:  Semi-direct products of {L}ie algebras, their invariants and representations. Publ. Res. Inst. Math. Sic. {\bf 4}, 1199--1257 (2007)

\bibitem{per} Perroud, M.: The fundamental invariants of inhomogeneous classical groups. J. Math. Phys. {\bf 24}, 1381--1391 (1983) 

\bibitem{duffy} Rico Mart\'{i}nez, J.M., Duffy, J.: The principle of transference: history, statement and proof. Mech. Mach. Theory {\bf 28}, 165--177 (1993)

\bibitem{rooney} Rooney, J.: William {K}ingdon {C}lifford (1845--1879). In M. Ceccarelli (ed.) Distinguished Figures in Mechanism and machine Science, pp. 79--116. Springer, Dordrecht (2007)

\bibitem{ros} Rosen, J.: Construction of invariants for {L}ie algebras of inhomogeneous pseudo-orthogonal and pseudo-unitary groups. J. Math. Phys. {\bf 9}, 1305--1307 (1968) 

\bibitem{Samuel} Samuel, A.E., McAree, P.R., Hunt, K.H.: Unifying screw geometry and matrix transformations. Int. J. Robotics Res. {\bf 10}, 454--472 (1991)

\bibitem{selig} Selig, J.:  Geometric fundamentals of robotics. Springer, New York (2005)

\bibitem{study} Study, E.: Geometrie der Dynamen. B.G. Teubner (1903)

\bibitem{tak} Takiff, S.J.: Invariant polynomials on Lie algebras of inhomogeneous unitary and special orthogonal groups. Trans. American Math. Soc. {\bf 170}, 221--230 (1972) 

\bibitem{weyl} Weyl, H.: The classical groups: their invariants and representations. Princeton University Press (1997)

\end{thebibliography}
\end{document}